\documentclass{amsart}
\usepackage[utf8]{inputenc}
\usepackage{amsmath}
\usepackage{amsthm}
\usepackage{amsfonts}
\usepackage{amssymb}
\usepackage[all]{xy}
\usepackage{hyperref}
\usepackage{indentfirst}

\newtheorem{thm}{Theorem}[section]
\newtheorem{prop}[thm]{Proposition}

\newtheorem{theorem}[thm]{Theorem}

\newtheorem{proposition}[thm]{Proposition}

\theoremstyle{definition}

\newtheorem{definition}[thm]{Definition}
\newtheorem*{ack}{Acknowledgement}
\newtheorem{remark}[thm]{Remark}

\title{Characterizations of locally finite actions of groups on sets}
\date{}
\author{Eduardo Scarparo}
\thanks{This work was supported by CNPq, National Council for Scientific and Technological Development - Brazil.}
\address{ Department of Mathematical Sciences, University of Copenhagen,\newline
Universitetsparken 5, DK-2100, Copenhagen, Denmark}
 \email{duduscarparo@gmail.com}
\begin{document}

\begin{abstract}

We show that an action of a group on a set $X$ is locally finite if and only if $X$ is not equidecomposable with a proper subset of itself. As a consequence, a group is locally finite if and only if its uniform Roe algebra is finite.

\end{abstract}
\maketitle
\section{Introduction}

Given a group acting on a set $X$, a property that has been well-studied is the existence of an invariant mean on $X$, that is, amenability of the action (see \cite{glasnermonod} for historical remarks). By Tarski's Theorem (\cite[Corollary 9.2]{s1993banach}), this is equivalent to $X$ not being equidecomposable with two disjoint subsets of itself.

 In this note, we address the following question: given an action of a group $G$ on a set $X$, when is $X$ not equidecomposable with a proper subset of itself? We show that this holds if and only if the action is locally finite (Definition \ref{lf}), if and only if $\ell^\infty(X)\rtimes_r G$ is a finite $\mathrm{C}^*$-algebra (Theorem \ref{inha}). It follows from this that a group is locally finite if and only if its uniform Roe algebra ($\ell^\infty(G)\rtimes_{\mathrm{r}}G$) is finite (Proposition \ref{roe}). In \cite{kmr}, it was shown that $\ell^\infty(G)\rtimes_{\mathrm{r}}G$ is finite if $G$ is locally finite and asked if the converse holds.

 It was already known that amenability of a group $G$ is equivalent to $\ell^\infty(G)\rtimes_{\mathrm{r}}G$ not being properly infinite, and supramenability is equivalent to $\ell^\infty(G)\rtimes_{\mathrm{r}}G$ not containing any properly infinite projections  (\cite[Proposition 5.3]{kmr}). Therefore, Proposition \ref{roe} completes the dictionary between equidecomposition properties of groups and the type of projections in the uniform Roe algebra.

\section{Characterizations of locally finite actions of groups on sets}
We start by recalling some definitions:

\begin{definition}
Let be $G$ be a group acting on a set $X$. Two subsets $A$ and $B$ of $X$ are said to be \textit{equidecomposable} if there are finite partitions $\{A_i\}_{i=i}^n$ and $\{B_i\}_{i=i}^n$ of $A$ and $B$, respectively, and elements $s_1,\dots,s_n\in G$ such that $B_i=s_iA_i$ for $1\leq i \leq n$. When we say that two subsets of $G$ are equidecomposable, it is with respect to the left action of $G$ on itself.
\end{definition}

The next definition has already been introduced in \cite{doi:10.1080/00927872.2015.1053903} for actions on semilattices.

\begin{definition}\label{lf}
An action of a group $G$ on a set $X$ is said to be \textit{locally finite} if, for every finitely generated subgroup $H$ of $G$ and every $x\in X$, the $H$-orbit of $x$ is finite.

\end{definition}

The left action of a group on itself is locally finite if and only if the group is locally finite.

The following result shows that the notion of locally finite action is a natural strengthening of the notion of amenable action on a set.

\begin{theorem}\label{inha}
Let $G$ be a group acting on a set $X$. The following conditions are equivalent:

\begin{enumerate}

\item The action is locally finite;

\item $\ell^\infty(X)\rtimes_r G$ is finite;

\item $X$ is not equidecomposable with a proper subset of itself;

\item No subset of $X$ is equidecomposable with a proper subset of itself.

\end{enumerate}

\end{theorem}
\begin{proof}

(1) $\Rightarrow$ (2). Since the inductive limit of finite unital $\mathrm{C}^*$-algebras with unital conneting maps is finite, it suffices to show that $\ell^\infty(X)\rtimes_{\mathrm{r}}H$ is finite for every finitely generated subgroup $H$ of $G$. Let $H$ be such a subgroup and $X=\sqcup_{i\in I} X_i$ be the partition of $X$ into its $H$-orbits.

 For every $i\in I$, the restriction map $\ell^\infty(X)\to\ell^\infty(X_i)$ is $H$-equivariant. Therefore, there is a homomorphism $\psi\colon\ell^\infty(X)\rtimes_\mathrm{r}H\to\prod(\ell^\infty(X_i)\rtimes_\mathrm{r}H)$. We claim that $\psi$ is injective. 

Let $\varphi\colon\ell^\infty(X)\rtimes_\mathrm{r}H\to\ell^\infty(X)$ and, for every $i\in I$, $\varphi_i\colon\ell^\infty(X_i)\rtimes_\mathrm{r}H\to\ell^\infty(X_i)$ be the canonical conditional expectations. Also let $\varphi_I\colon\prod(\ell^\infty(X_i)\rtimes_\mathrm{r}H)\to\prod\ell^\infty(X_i)$ be the product of the maps $\varphi_i$, and $T\colon\ell^\infty(X)\to\prod\ell^\infty(X_i)$ be the isomorphism which arises from the product of the restriction maps. The following diagram commutes:

\begin{equation*}
\begin{xymatrix}
{\ell^\infty(X)\rtimes_\mathrm{r}H\ar^{\psi}[r]\ar_{\varphi}[d]&\prod(\ell^\infty(X_i)\rtimes_\mathrm{r}H)\ar^{\varphi_I}[d]\\
\ell^\infty(X)\ar_{T}[r]&\prod\ell^\infty(X_i).
}
\end{xymatrix}
\end{equation*}
Since $\varphi$ is faithful, we conclude that $\psi$ is injective. Since the product of finite unital $\mathrm{C}^*$-algebras is finite, it suffices to show that $\ell^\infty(X_i)\rtimes_\mathrm{r}H$ is finite for every $i\in I$ in order to conclude that $\ell^\infty(X)\rtimes_\mathrm{r}H$ is finite.

Given $i\in I$, let $\tau_i$ be the tracial state on $\ell^\infty(X_i)$ which arises from the uniform probability measure on the finite set $X_i$. Since $\tau_i$ is $H$-invariant and faithful, it follows that the map $\tau_i\circ\varphi_i\colon\ell^\infty(X_i)\rtimes_\mathrm{r}H\to\mathbb{C}$ is a faithful tracial state. Therefore, $\ell^\infty(X_i)\rtimes_\mathrm{r}H$ is finite.

(2) $\Rightarrow$ (3). This follows from the fact that, if $A$ and $B$ are equidecomposable subsets of $X$, then the projections $1_A$ and $1_B$ are equivalent in $\ell^\infty(X)\rtimes_\mathrm{r} G$.

(3) $\Rightarrow$ (4). If $A\subset X$ is equidecomposable with $B\subsetneq A$, then $X=A\sqcup A^c$ is equidecomposable with $B\sqcup A^c\subsetneq X$.

(4) $\Rightarrow$ (1). Suppose that there is $H<G$ generated by a finite and symmetric set $S$ and $x\in X$ such that $Hx$ is infinite. Then there exists a sequence $(s_n)_{n\in\mathbb{N}}\subset S$ such that 

$$\forall n,m\in\mathbb{N} \colon n\neq m\Rightarrow s_n\cdots s_1x\neq s_m\cdots s_1x.$$

 The sequence $(s_n\cdots s_1x)_{n\in\mathbb{N}}$ can be seen as an infinite simple path in the graph whose vertex set is $Hx$ and whose edges come from $S$.

 We claim that $\gamma:=\{s_n\cdots s_1x\colon n\in\mathbb{N}\}$ is equidecomposable with $\gamma\setminus\{s_1x\}$. 

Given $s\in S$, let $A_s:=\{s_n\cdots s_1x\colon s_{n+1}=s\}$. It is easy to check that $\{A_s\}_{s\in S}$ partitions $\gamma$ and $\{sA_s\}_{s\in S}$ partitions $\gamma\setminus\{s_1x\}$. Hence, $\gamma$ is equidecomposable with its proper subset $\gamma\setminus\{s_1x\}$.

\end{proof}

We now proceed to give a characterization of locally finite groups which can be seen as an analogy to parts of \cite[Theorem 1.1]{kmr}.

The next definition is from \cite{shalom2004harmonic}.

\begin{definition}
Let $H$ and $G$ be groups. A map $f\colon H\to G$ is said to be a \textit{uniform embedding} if, for every finite set $S\subset H$, there is a finite set $T\subset G$ such that:

$$\forall x,y\in H \colon xy^{-1}\in S\implies f(x)f(y)^{-1}\in T$$
and, for every finite set $T\subset G$, there is $S\subset H$ finite such that
$$\forall x,y\in H \colon f(x)f(y)^{-1}\in T\implies xy^{-1}\in S.$$
\end{definition}

The implication (1) $\Rightarrow$ (2) in the next result had already been observed in \cite[Remark 5.4]{kmr}, and (5) $\Rightarrow$ (1) is an immediate consequence of \cite[Lemma 1]{zuk2000isoperimetric}.

\begin{proposition}\label{roe}
Let $G$ be a group. The following conditions are equivalent:
\begin{enumerate}
\item$G$ is locally finite;
\item The uniform Roe algebra $\ell^\infty(G)\rtimes_{\mathrm{r}}G$ is finite;
\item $G$ is not equidecomposable with a proper subset of itself;
\item No subset $A\subset G$ is equidecomposable with a proper subset of itself;
\item There is no injective uniform embedding from $\mathbb{Z}$ into $G$.
\end{enumerate}

\end{proposition}
\begin{proof}

The implications (1) $\Rightarrow$ (2) $\Rightarrow$ (3) $\Rightarrow$ (4) (and (4) $\Rightarrow$ (1)) are a consequence of Theorem \ref{inha}. 

(4) $\Rightarrow$ (5). This follows from the fact that $\mathbb{N}\subset \mathbb{Z}$ is equidecomposable with a proper subset of itself and \cite[Lemma 3.2]{kmr}.

(5) $\Rightarrow$ (1). This is a consequence of \cite[Lemma 1]{zuk2000isoperimetric}.
\end{proof}

\begin{remark}
After this note was made available on arXiv, we became aware of \cite{pnc}, where it is shown that if a group is infinite and finitely generated, then its uniform Roe algebra is infinite.
\end{remark}

Any locally finite group acts on itself in a transitive, faithful and locally finite way. If a finitely generated group admits a faithful, transitive, locally finite action on a set, then the group is finite. This is in stark contrast to the fact that there are finitely generated, non-amenable groups which admit faithful, transitive, amenable actions on sets (see \cite{glasnermonod} for various examples).

A finitely generated group admits a faithful, locally finite action if and only if it is residually finite.

\begin{proposition}
If a group admits a faithful, locally finite action, then it embeds into a group which admits a faithful, locally finite and transitive action.
\end{proposition}
\begin{proof}

Let $G$ be a group which acts on a set $X$ in a faithful and locally finite way. 

Take a set $Y\subset X$ of representatives of all $G$-orbits, and let $S_Y$ be the group of finitely supported permutations of $Y$. Consider the natural action of $S_Y$ on $X$ and the associated action of $H:=G*S_Y$ on $X$. This action is transitive and locally finite. By taking the quotient of $H$ by the kernel of this action, we get a faithful, transitive, locally finite action on $X$ by a group which contains $G$.

\end{proof}

In analogy to what is known for amenable actions (\cite[Lemma 3.2]{juschenko2013cantor}), the following holds for locally finite actions:
\begin{prop}
Let $G$ be a group acting on a set $X$ in a locally finite way. If, for every $x\in X$, the stabilizer subgroup $G_x$ is locally finite, then $G$ is locally finite.
\end{prop}
\begin{proof}
Take $H<G$ finitely generated and $x\in X$. Since the action is locally finite, it follows that $Hx$ is finite. Hence, there is $H_0$ a subgroup of finite index in $H$ such that $H_0<G_x$. In particular, $H_0$ is locally finite. Therefore, $H$ is also locally finite. Since $H$ is finitely generated, we conclude that it is finite.
\end{proof}
\begin{remark}
One can define in a natural way an action of a group on a set $X$ to be supramenable if no subset of $X$ is equidecomposable with two disjoint proper subsets of itself. It is not true that if the action of a group $G$ is supramenable, and all the stabilizer subgroups are supramenable, then $G$ is supramenable. 

Indeed, it is well-known that the class of supramenable groups is not closed by taking extensions (the lamplighter group $\mathbb{Z}_2\wr\mathbb{Z}$ is such an example). Let then $G$ be a non-supramenable group which contains a supramenable normal subgroup $N$ such that $\frac{G}{N}$ is also supramenable. 

Consider the left action of $G$ on $\frac{G}{N}$. Since $\frac{G}{N}$ is supramenable, it follows easily that this action is supramenable. The stabilizer subgroups of the action are all equal to $N$, hence are supramenable.
\end{remark}

\begin{ack}
The author thanks Claire Anantharaman-Delaroche for calling his attention to the reference \cite{glasnermonod}.
\end{ack}

\bibliographystyle{acm}
\bibliography{bibliografia}

\begin{thebibliography}{1}

\bibitem{glasnermonod}
{\sc Glasner, Y., and Monod, N.}
\newblock Amenable actions, free products and a fixed point property.
\newblock {\em Bulletin of the London Mathematical Society 39}, 1 (2007),
  138--150.

\bibitem{juschenko2013cantor}
{\sc Juschenko, K., and Monod, N.}
\newblock Cantor systems, piecewise translations and simple amenable groups.
\newblock {\em Annals of Mathematics 178}, 2 (2013), 775--787.

\bibitem{kmr}
{\sc Kellerhals, J., Monod, N., and R{\o}rdam, M.}
\newblock Non-supramenable groups acting on locally compact spaces.
\newblock {\em Documenta Mathematica 18\/} (2013), 1597--1626.

\bibitem{shalom2004harmonic}
{\sc Shalom, Y.}
\newblock Harmonic analysis, cohomology, and the large-scale geometry of
  amenable groups.
\newblock {\em Acta Mathematica 192}, 2 (2004), 119--185.

\bibitem{doi:10.1080/00927872.2015.1053903}
{\sc Silva, P.~V., and Soares, F.}
\newblock Howson's property for semidirect products of semilattices by groups.
\newblock {\em Communications in Algebra 44}, 6 (2016), 2482--2494.

\bibitem{s1993banach}
{\sc Wagon, S.}
\newblock {\em The Banach-Tarski Paradox}, vol.~24.
\newblock Cambridge University Press, 1993.

\bibitem{pnc}
{\sc Wei, S.}
\newblock On the quasidiagonality of {R}oe algebras.
\newblock {\em Science China Mathematics 54}, 5 (2011), 1011--1018.

\bibitem{zuk2000isoperimetric}
{\sc {\.Z}uk, A.}
\newblock On an isoperimetric inequality for infinite finitely generated
  groups.
\newblock {\em Topology 39}, 5 (2000), 947--956.

\end{thebibliography}
\end{document}